\documentclass[a4paper,11pt]{amsart}
\usepackage[T1]{fontenc}
 \makeatletter \theoremstyle{plain}
\newtheorem{thm}{Theorem}[section]
\numberwithin{equation}{section}
\numberwithin{figure}{section} 
\theoremstyle{plain}
\newtheorem{cor}[thm]{Corollary} 
\theoremstyle{plain}
\newtheorem{lem}[thm]{Lemma} 
\theoremstyle{plain}
\theoremstyle{plain}

\begin{document}

\title{DOMAIN OF DIFFERENCE MATRIX OF ORDER ONE IN SOME SPACES OF DOUBLE SEQUENCES}

\author{serkan dem\.{I}r\.{I}z and osman duyar}
\curraddr{Gaziosmanpa\c{s}a University, Faculty of Arts and Science \\
Department of Mathematics, Tokat-60240/TURKEY} \email{serkandemiriz@gmail.com\\
osman-duyar@hotmail.com}

\begin{abstract}
In this study, we define the spaces
$\mathcal{M}_{u}(\Delta),\mathcal{C}_{p}(\Delta),\mathcal{C}_{0p}(\Delta),
\mathcal{C}_{r}(\Delta)$ and $\mathcal{L}_{q}(\Delta)$ of double
sequences whose difference transforms are bounded , convergent in
the Pringsheim's sense, null in the Pringsheim's sense, both
convergent in the Pringsheim's sense and bounded, regularly
convergent and absolutely $q-$summable, respectively, and also
examine some inclusion relations related to those sequence spaces.
Furthermore, we show that these sequence spaces are Banach spaces  .
We determine the alpha-dual of the space $\mathcal{M}_{u}(\Delta)$
and the $\beta(v)-$dual of the space $\mathcal{C}_{\eta}(\Delta)$ of
double sequences, where $v,\eta\in \{p,bp,r\}$. Finally, we
characterize the classes $(\mu:\mathcal{C}_{v}(\Delta))$ for $v\in
\{p,bp,r\}$ of four dimensional matrix transformations, where $\mu$
is any given space of double sequences.
\end{abstract} \maketitle

\section{Introduction}

By $\omega$ and $\Omega$, we denote the sets of all real valued
single and double sequences which are the vector spaces with
coordinatewise addition and scalar multiplication. Any vector
subspaces of  $\omega$ and $\Omega$ are called as the \emph{single
sequence space} and \emph{double sequence space}, respectively. By
$\mathcal{M}_{u}$, we denote the space of all bounded double
sequences, that is
$$
\mathcal{M}_{u}:=\bigg\{x=(x_{mn})\in\Omega: \|x\|_{\infty}=
\sup_{m,n\in \mathbb{N}} |x_{mn}|<\infty \bigg\}
$$
which is a Banach space with the norm $\|x\|_{\infty}$; where
$\mathbb{N}$ denotes the set of  all positive integers. Consider a
sequence $x=(x_{mn})\in\Omega$. If for every $\varepsilon >0$ there
exists $n_{0}=n_{0}(\varepsilon)\in \mathbb{N}$ and $l\in
\mathbb{R}$ such that $|x_{mn}-l|<\varepsilon$ for all $m,n>n_{0}$
then we call that the double sequence $x$ is \emph{convergent} in
the \emph{Pringsheim's sense} to the limit $l$ and write $p-\lim
x_{mn}=l$; where $\mathbb{R}$ denotes the real field. By
$\mathcal{C}_{p}$, we denote the space of all convergent double
sequences in the Pringsheim's sense. It is well-known that there are
such sequences in the space $\mathcal{C}_{p}$ but not in the space
$\mathcal{M}_{u}$. Indeed following Boos \cite{boos}, if we define
the sequence $x=(x_{mn})$ by

\begin{eqnarray*}
x_{mn}:=\left\{
\begin{array}{rcl}
n&, & m=1, n\in \mathbb{N}, \\
0&,  &m\geq 2, n\in \mathbb{N}, \\
\end{array}
\right.
\end{eqnarray*}
then it is trivial that $x\in \mathcal{C}_{p}\backslash
\mathcal{M}_{u}$, since $p-\lim x_{mn}=0$ but
$\|x\|_{\infty}=\infty$. So, we can consider the space
$\mathcal{C}_{bp}$ of the double sequences which are both convergent
in the Pringsheim's sense and bounded, i.e.,
$\mathcal{C}_{bp}=\mathcal{C}_{p}\cap \mathcal{M}_{u}$. A sequence
in the space $\mathcal{C}_{p}$ is said to be \emph{regularly
convergent} if it is a single convergent sequence with respect to
each index and denote the set of all such sequences by
$\mathcal{C}_{r}.$ Also by $\mathcal{C}_{bp0}$ and
$\mathcal{C}_{r0}$, we denote the spaces of all double sequences
converging to $0$ contained in the sequence spaces
$\mathcal{C}_{bp}$ and $\mathcal{C}_{r}$, respectively. M\'{o}ricz
\cite{fm} proved that
$\mathcal{C}_{bp},\mathcal{C}_{bp0},\mathcal{C}_{r}$ and
$\mathcal{C}_{r0}$ are Banach spaces with the norm $\|.\|_{\infty}$.

Let us consider the isomorphism $T$ defined by Zeltser \cite{zel1}
as
\begin{eqnarray}
T&:&\Omega\rightarrow \omega\\
&&x\mapsto z=(z_{i}):=(x_{\varphi^{-1}(i)}),\nonumber
\end{eqnarray}
where $\varphi:\mathbb{N}\times \mathbb{N}\rightarrow \mathbb{N}$ is
a bijection defined by

\begin{eqnarray*}
\varphi [(1,1)]&=&1,\\
\varphi [(1,2)]&=&2, \varphi [(2,2)]=3, \varphi [(2,1)]=4, \\
&.& \\
&.& \\
&.& \\
\varphi [(1,n)]&=& (n-1)^{2}+1, \varphi [(2,n)]=(n-1)^{2}+2,...,\\
\varphi [(n,n)]&=& (n-1)^{2}+n, \varphi [(n,n-1)]=
n^{2}-n+2,...,\varphi [(n,1)]=n^{2},\\
&.& \\
&.& \\
&.& \\
\end{eqnarray*}

Let us consider a double sequence $x=(x_{mn})$ and define the
sequence $s=(s_{mn})$ which will be used throughout via $x$ by

\begin{equation}
s_{mn}:=\sum_{i=0}^{m}\sum_{j=0}^{n}x_{ij}
\end{equation}
for all $m,n\in \mathbb{N}$ . For the sake of brevity, here and in
what follows, we abbreviate the summation $\sum_{i=0}^{\infty}
\sum_{j=0}^{\infty}$ by $\sum_{i,j}$ and we use this abbreviation
with other letters. Let $\lambda$ be a space of a double sequences,
converging with respect to some linear convergence rule
$v-\lim:\lambda\rightarrow \mathbb{R}$. The sum of a double series
$\sum_{i,j} x_{ij}$ with respect to this rule is defined by
$v-\sum_{i,j} x_{ij}=v-\lim_{m,n\rightarrow \infty} s_{mn}$. Let
$\lambda,\mu$ be two spaces of double sequences, converging with
respect to the linear convergence rules $v_{1}-\lim$ and
$v_{2}-\lim$, respectively, and $A=(a_{mnkl})$ also be a four
dimensional infinite matrix over the real or complex field.

The \emph{$\alpha-$dual} $\lambda^{\alpha}$, \emph{$\beta(v)-$dual}
$\lambda^{\beta(v)}$ with respect to the $v-$convergence for $v\in
\{p,bp,r\}$ and the \emph{$\gamma-$dual} $\lambda^{\gamma}$ of a
double sequence space $\lambda$ are respectively defined by
\begin{eqnarray*}
\lambda^{\alpha}&:=&\bigg\{(a_{ij})\in \Omega: \sum_{i,j}
|a_{ij}x_{ij}|<\infty \quad \textrm{for all} \quad (x_{ij})\in
\lambda \bigg\},\\
\lambda^{\beta(v)}&:=&\bigg\{(a_{ij})\in \Omega: v-\sum_{i,j}
a_{ij}x_{ij} \quad \textrm{exists for all} \quad (x_{ij})\in \lambda
\bigg\},\\
\lambda^{\gamma}&:=&\bigg\{(a_{ij})\in \Omega: \sup_{k,l\in
\mathbb{N}} \bigg|\sum_{i,j=1}^{k,l}a_{ij}x_{ij}\bigg|<\infty \quad
\textrm{for all} \quad (x_{ij})\in \lambda \bigg\}.
\end{eqnarray*}
It is easy to see for any two spaces $\lambda,\mu$ of double
sequences that $\mu^{\alpha}\subset \lambda^{\alpha}$ whenever
$\lambda \subset \mu$ and $\lambda^{\alpha}\subset
\lambda^{\gamma}$. Additionally, it is known that the inclusion
$\lambda^{\alpha}\subset \lambda^{\beta(v)}$ holds while the
inclusion $\lambda^{\beta(v)} \subset \lambda^{\gamma}$ does not
hold, since the $v-$convergence of the sequence of partial sums of a
double series does not imply its boundedness.

The \emph{$v-$summability domain} $\lambda_{A}^{(v)}$  of a four
dimensional infinite matrix $A=(a_{mnkl})$ in a space $\lambda$ of a
double sequences is defined by
\begin{equation}\label{1.3}
\lambda_{A}^{(v)}=\bigg\{x=(x_{kl})\in \Omega:
Ax=\bigg(v-\sum_{k,l}a_{mnkl}x_{kl}\bigg)_{m,n\in \mathbb{N}} \
\textrm{exists and is in } \ \lambda \bigg\}.
\end{equation}
We say, with the notataion (\ref{1.3}), that $A$ maps the space
$\lambda$ into the space $\mu$ if and only if $Ax$ exists and is in
$\mu$ for all $x\in \lambda$ and denote the set of all four
dimensional matrices, transforming the space $\lambda$ into the
space $\mu$, by $(\lambda:\mu)$. It is trivial that for any matrix
$A\in (\lambda:\mu)$, $(a_{mnkl})_{k,l\in \mathbb{N}}$ is in the
$\beta(v)-$dual $\lambda^{\beta(v)}$ of the space $\lambda$ for all
$m,n\in \mathbb{N}$. An infinite matrix $A$ is said to be
\emph{$\mathcal{C}_{v}-$conservative} if $\mathcal{C}_{v}\subset
(\mathcal{C}_{v})_{A}$. Also by $(\lambda:\mu;p)$, we denote the
class of all four dimensional matrices $A=(a_{mnkl})$ in the class
$(\lambda:\mu)$ such that $v_{2}-\lim Ax=v_{1}-\lim x$ for all $x\in
\lambda$.

Now, following Zeltser \cite{zel2} we note the terminology for
double sequence spaces. A locally convex double sequence space
$\lambda$ is called a \emph{$DK-$space}, if all of the seminorms
$r_{kl}:\lambda\rightarrow \mathbb{R}$, $x=(x_{kl})\mapsto |x_{kl}|$
for all $k,l\in \mathbb{N}$ are continuous. A $DK-$space with a
Fr\'{e}chet topology is called an \emph{$FDK-$space}. A normed
$FDK-$space is called a \emph{$BDK-$space}. We record that $C_{r}$
endowed with the norm $\|.\|_{\infty}:\mathcal{C}_{r}\rightarrow
\mathbb{R}$, $x=(x_{kl})\mapsto \sup_{k,l\in \mathbb{N}} |x_{kl}|$
is a $BDK-$space.

Let us define the following sets of double sequences:
\begin{eqnarray*}
\mathcal{M}_{u}(t)&:=&\bigg\{(x_{mn})\in \Omega: \sup_{m,n\in
\mathbb{N}} |x_{mn}|^{t_{mn}}<\infty \bigg\},\\
\mathcal{C}_{p}(t)&:=&\bigg\{(x_{mn})\in \Omega: \exists l\in
\mathbb{C}\ni p-\lim_{m,n\rightarrow \infty}
|x_{mn}-l|^{t_{mn}}=0\bigg\},\\
\mathcal{C}_{0p}(t)&:=& \bigg\{(x_{mn})\in \Omega:
p-\lim_{m,n\rightarrow \infty} |x_{mn}|^{t_{mn}}=0\bigg\},\\
\mathcal{L}_{u}(t)&:=&\bigg\{(x_{mn})\in \Omega: \sum_{m,n}
|x_{mn}|^{t_{mn}}<\infty\bigg\}, \\
\mathcal{C}_{bp}(t)&:=&\mathcal{C}_{p}(t)\cap \mathcal{M}_{u}(t)
\quad  \textrm{and} \quad
\mathcal{C}_{0bp}(t):=\mathcal{C}_{0p}(t)\cap \mathcal{M}_{u}(t);
\end{eqnarray*}
where $t=(t_{mn})$ is the sequence of strictly positive reals
$t_{mn}$ for all $m,n\in \mathbb{N}$. In the case $t_{mn}=1$ for all
$m,n\in \mathbb{N}$;
$\mathcal{M}_{u}(t),\mathcal{C}_{p}(t),\mathcal{C}_{0p}(t),\mathcal{L}_{u}(t),\mathcal{C}_{bp}(t)$
and $\mathcal{C}_{0bp}(t)$ reduce to the sets
$\mathcal{M}_{u},\mathcal{C}_{p},\mathcal{C}_{0p},\mathcal{L}_{u},\mathcal{C}_{bp}$
and $\mathcal{C}_{0bp}$, respectively. Now, we can summarize the
knowledge given in some document related to the double sequence
spaces. Gökhan and Çolak \cite{agrc1,agrc2} have proved that
$\mathcal{M}_{u}(t),\mathcal{C}_{p}(t)$ and $\mathcal{C}_{bp}(t)$
are complete paranormed spaces of double sequences and gave the
alpha-, beta-, gamma-duals of the spaces $\mathcal{M}_{u}(t)$ and
$\mathcal{C}_{bp}(t)$. Quite recently, in her PhD thesis, Zeltser
\cite{zel1} has essentially studied  both the theory of topological
double sequence spaces and the theory of summability of double
sequences. Mursaleen and Edely \cite{mme} have introduced the
statistical convergence and statistical Cauchy for double sequences,
and gave the relation between statistically convergent and strongly
Ces\`{a}ro summable double sequences. Nextly, Mursaleen \cite{mm1}
and Mursaleen and Edely \cite{mme2} have defined the almost strong
regularity of matrices for double sequences and applied these
matrices to establish a core theorem and introduced the $M-$core for
double sequences and determined those four dimensional matrices
transforming every bounded double sequence $x=(x_{jk})$ into one
whose core is a subset of the $M-$core of $x$. More recently, Altay
and Ba\c{s}ar \cite{bafb1} have defined  the spaces
$\mathcal{BS},\mathcal{BS}(t), \mathcal{CS}_{bp}, \mathcal{CS}_{r}$
and $\mathcal{BV}$ of double series whose sequence of partial sums
are in the spaces $\mathcal{M}_{u}, \mathcal{M}_{u}(t),
\mathcal{C}_{p}, \mathcal{C}_{bp}, \mathcal{C}_{r}$ and
$\mathcal{L}_{u}$, respectively, and also examined some properties
of those sequence spaces and determined the alpha-duals of the
spaces $\mathcal{BS}, \mathcal{BV}, \mathcal{CS}_{bp}$ and the
$\beta (v)-$duals of the spaces $\mathcal{CS}_{bp}$ and
$\mathcal{CS}_{r}$ of double series. Quite recently, Ba\c{s}ar and
Sever \cite{fbys} have introduced the Banach space $\mathcal{L}_{q}$
of double sequences corresponding to the well-known space $\ell_{q}$
of absolutely $q-$summable single sequences and examine some
properties of the space $\mathcal{L}_{q}$. Furthermore, they
determine the $\beta(v)-$dual of the space and establish that the
alpha- and gamma-duals of the space $\mathcal{L}_{q}$ coincide with
the $\beta(v)-$dual; where
\begin{eqnarray*}
\mathcal{L}_{q}&:=&\bigg\{(x_{ij})\in \Omega: \sum_{i,j}
|x_{ij}|^{q}<\infty \bigg\}, \quad (1\leq q<\infty),\\
\mathcal{CS}_{v}&:=&\big\{(x_{ij})\in \Omega: (s_{mn})\in
\mathcal{C}_{v}\big\}.
\end{eqnarray*}
Here and after we assume that $v\in \{p,bp,r\}$.

The double difference matrix $\Delta=(\delta_{mnkl})$ of order one
is defined by

\begin{eqnarray*}
\delta_{mnkl}:=\left\{
\begin{array}{rcl}
(-1)^{m+n-k-l}&, & m-1\leq k\leq m, \quad n-1\leq l\leq n, \\
0&,  &\textrm{otherwise} \\
\end{array}
\right.
\end{eqnarray*}
for all $m,n,k,l\in \mathbb{N}$. Define the sequence $y=(y_{mn})$ as
the $\Delta-$transform of a sequence $x=(x_{mn})$, i.e.,
\begin{equation}
y_{mn}=(\Delta x)_{mn}=x_{mn}-x_{m,n-1}-x_{m-1,n}+x_{m-1,n-1}
\end{equation}
for all $m,n\in \mathbb{N}$. Additionally, a direct calculation
gives the inverse $\Delta^{-1}=S=(s_{mnkl})$ of the matrix $\Delta$
as follows:

\begin{eqnarray*}
s_{mnkl}:=\left\{
\begin{array}{rcl}
1&, &0\leq k\leq m, \quad 0\leq l\leq n, \\
0&,  &\textrm{otherwise} \\
\end{array}
\right.
\end{eqnarray*}
for all $m,n,k,l\in \mathbb{N}$.

In the present paper, we introduce the new double difference
sequence spaces $\mathcal{M}_{u}(\Delta), \mathcal{C}_{p}(\Delta),
\mathcal{C}_{0p}(\Delta)$ and $\mathcal{L}_{q}(\Delta)$, that is,
\begin{eqnarray*}
\mathcal{M}_{u}(\Delta)&:=&\big\{(x_{mn})\in \Omega: \sup_{m,n\in
\mathbb{N}} |y_{mn}|<\infty\big\},\\
\mathcal{C}_{p}(\Delta)&:=&\big\{(x_{mn})\in \Omega:
\exists l\in \mathbb{C}\ni p-\lim_{m,n\rightarrow \infty} |y_{mn}-l|=0\big\},\\
\mathcal{C}_{0p}(\Delta)&:=&\big\{(x_{mn})\in \Omega:
p-\lim_{m,n\rightarrow \infty} |y_{mn}|=0\big\},\\
\mathcal{L}_{q}(\Delta)&:=&\bigg\{(x_{ij})\in \Omega: \sum_{m,n}
|y_{mn}|^{q}<\infty\bigg\}, \quad (1\leq q<\infty).
\end{eqnarray*}
By $\mathcal{C}_{bp}(\Delta)$ and $\mathcal{C}_{r}(\Delta)$, we
denote the sets of all the $\Delta-$transforms convergent and
bounded, and the $\Delta-$transforms regularly convergent double
sequences. One can easily see that the spaces
$\mathcal{M}_{u}(\Delta), \mathcal{C}_{p}(\Delta),$
$\mathcal{C}_{0p}(\Delta),  \mathcal{C}_{bp}(\Delta),
\mathcal{C}_{r}(\Delta)$ and $\mathcal{L}_{q}(\Delta)$ are the
domain of the double difference matrix $\Delta$ in the spaces
$\mathcal{M}_{u}, \mathcal{C}_{p}, \mathcal{C}_{0p},
\mathcal{C}_{bp},$ $\mathcal{C}_{r}$ and $\mathcal{L}_{q}$,
respectively.

\section{Some new double difference sequence spaces}

In the present section, we deal with the sets
$\mathcal{M}_{u}(\Delta), \mathcal{C}_{p}(\Delta),$
$\mathcal{C}_{0p}(\Delta),  \mathcal{C}_{bp}(\Delta),
\mathcal{C}_{r}(\Delta)$ and $\mathcal{L}_{q}(\Delta)$ consisting of
the double sequences whose $\Delta-$transforms of order one are in
the spaces $\mathcal{M}_{u}, \mathcal{C}_{p}, \mathcal{C}_{0p},
\mathcal{C}_{bp},$ $\mathcal{C}_{r}$ and $\mathcal{L}_{q}$,
respectively.

\begin{thm}
The sets $\mathcal{M}_{u}(\Delta), \mathcal{C}_{p}(\Delta),$
$\mathcal{C}_{0p}(\Delta),  \mathcal{C}_{bp}(\Delta),
\mathcal{C}_{r}(\Delta)$ and $\mathcal{L}_{q}(\Delta)$ are the
linear spaces with the coordinatewise addition and scalar
multiplication, and $\mathcal{M}_{u}(\Delta),
\mathcal{C}_{p}(\Delta),$ $\mathcal{C}_{0p}(\Delta),
\mathcal{C}_{bp}(\Delta), \mathcal{C}_{r}(\Delta)$ and
$\mathcal{L}_{q}(\Delta)$ are the Banach spaces with the norms
\begin{equation}
\|x\|_{\mathcal{M}_{u}(\Delta)}=\sup_{m,n\in \mathbb{N}}
\big|x_{mn}+x_{m-1,n-1}-x_{m,n-1}-x_{m-1,n}\big|,
\end{equation}
\begin{equation}
\|x\|_{\mathcal{L}_{q}(\Delta)}=\bigg[\sum_{m,n}
|x_{mn}+x_{m-1,n-1}-x_{m,n-1}-x_{m-1,n}|^{q}\bigg]^{1/q}, \quad
(1\leq q<\infty).
\end{equation}
\end{thm}
\begin{proof}
The first part of the theorem is a routine verification. So, we omit
the detail.

Since the proof may be given for the spaces
$\mathcal{M}_{u}(\Delta), \mathcal{C}_{p}(\Delta),$
$\mathcal{C}_{0p}(\Delta),  \mathcal{C}_{bp}(\Delta)$ and
$\mathcal{C}_{r}(\Delta)$, to avoid the repetition of the similar
statements, we prove  the theorem only for the space
$\mathcal{L}_{q}(\Delta)$.

It is obvious that $\|x\|_{\mathcal{L}_{q}(\Delta)}=\|y\|_{q}$,
where $\|.\|_{q}$ is the norm on the space $\mathcal{L}_{q}$. Let
$x^{(r)}=\{x_{jk}^{(r)}\}$ be a Cauchy sequence in
$\mathcal{L}_{q}(\Delta)$. Then, $\{y^{(r)}\}_{r\in \mathbb{N}}$ is
a Cauchy sequence in $\mathcal{L}_{q}$, where
$y^{(r)}=\{y_{mn}^{(r)}\}_{m,n=0}^{\infty}$ with
$$
y_{mn}^{(r)}=x_{mn}^{(r)}+x_{m-1,n-1}^{(r)}-x_{m,n-1}^{(r)}-x_{m-1,n}^{(r)}
$$
for all $m,n,r\in \mathbb{N}$. Then, for a given $\varepsilon>0$,
there is a positive integer $N=N(\varepsilon)$ such that
\begin{equation}\label{2.3}
\|y^{(r)}-y^{(s)}\|_{q}=\bigg\{\sum_{m,n}
\big|y_{mn}^{(r)}-y_{mn}^{(s)}\big|^{q}\bigg\}^{1/q}<\varepsilon
\end{equation}
for all $r,s>N$. Therefore
$\big|y_{mn}^{(r)}-y_{mn}^{(s)}\big|<\varepsilon$, i.e.
$\{y_{mn}^{(r)}\}_{r\in \mathbb{N}}$ is a Cauchy sequence in
$\mathbb{C}$, and hence converges in $\mathbb{C}$. Say,
\begin{equation}\label{2.4}
\lim_{r\rightarrow \infty}y_{mn}^{(r)}=y_{mn}.
\end{equation}
Using these infinitely many limits, we define the sequence
$y=(y_{mn})_{m,n=0}^{\infty}$. Then, we get by (\ref{2.4}) that
\begin{equation}\label{2.5}
\lim_{r\rightarrow \infty}
\|y_{mn}^{(r)}-y_{mn}\|_{\mathcal{L}_{q}(\Delta)}=\lim_{r\rightarrow
\infty}\bigg\{\sum_{m,n}\big|y_{mn}-y_{mn}^{(r)}\big|^{q}\bigg\}^{1/q}=0.
\end{equation}
Now, we have to show that $y\in \mathcal{L}_{q}$. Since
$y^{(r)}=\{y_{mn}^{(r)}\}_{m,n=0}^{\infty}\in \mathcal{L}_{q}$ and
by (\ref{2.5})
$$
\bigg\{\sum_{m,n}|y_{mn}|^{q}\bigg\}^{1/q}\leq
\bigg\{\sum_{m,n}|y_{mn}-y_{mn}^{(r)}|^{q}\bigg\}^{1/q}+\bigg\{\sum_{m,n}|y_{mn}^{(r)}|^{q}\bigg\}^{1/q}<
\infty,
$$
which shows that $y=(y_{mn})_{m,n=0}^{\infty}\in \mathcal{L}_{q}$,
i.e. $x=(x_{jk})\in \mathcal{L}_{q}(\Delta)$. Since
$\{x^{(r)}\}_{r\in \mathbb{N}}$ was arbitrary Cauchy sequence in
$\mathcal{L}_{q}(\Delta)$, the double difference sequence space
$\mathcal{L}_{q}(\Delta)$ is complete. This completes the proof of
the theorem.
\end{proof}

\begin{thm}\label{t2.2}
The space $\lambda (\Delta)$ is linearly isomorphic to the space
$\lambda$, where $\lambda$ denotes any of the spaces
$\mathcal{M}_{u}, \mathcal{C}_{p}, \mathcal{C}_{0p},
\mathcal{C}_{bp}, \mathcal{C}_{r}$ and $\mathcal{L}_{q}$.
\end{thm}
\begin{proof}
We show here that $\mathcal{M}_{u}(\Delta)$ is linearly isomorphic
to $\mathcal{M}_{u}$. Consider the transformation $T$ from
$\mathcal{M}_{u}(\Delta)$ to $\mathcal{M}_{u}$ defined by
$x=(x_{jk})\mapsto y=(y_{mn})$. Then, clearly $T$ is linear and
injective. Now, define the sequence $x=(x_{jk})$ by
\begin{equation}\label{2.6}
x_{jk}=\sum_{m=0}^{j}\sum_{n=0}^{k}y_{mn}
\end{equation}
for all $j,k\in \mathbb{N}$. Suppose that $y\in \mathcal{M}_{u}$.
Then, since
\begin{eqnarray*}
\|x\|_{\mathcal{M}_{u}(\Delta)}&=&\sup_{j,k\in \mathbb{N}}
\bigg|\sum_{m=0}^{j}\sum_{n=0}^{k}y_{mn}\bigg|\\
&=&\sup_{j,k\in \mathbb{N}} |y_{jk}|=\|y\|_{\infty}<\infty,
\end{eqnarray*}
$x=(x_{jk})$ defined by (\ref{2.6}) is in the space
$\mathcal{M}_{u}(\Delta)$. Hence, $T$ is surjective and norm
preserving. This completes the proof of the theorem.
\end{proof}

Now, we give some inclusion relations between the double difference
sequence spaces.
\begin{thm}\label{t2.3}
$\mathcal{M}_{u}$ is the subspace of the space
$\mathcal{M}_{u}(\Delta)$.
\end{thm}
\begin{proof}
 Let us take $x=(x_{mn})\in
\mathcal{M}_{u}$. Then, there exists an $K$ such that
$$
\sup_{m,n\in \mathbb{N}} |x_{mn}|\leq K
$$
for all $m,n\in \mathbb{N}$, one can observe that

\begin{eqnarray}
\big|(\Delta x)_{mn}\big|&=&\big|x_{mn}-x_{m,n-1}-x_{m-1,n}+x_{m-1,n-1}\big| \nonumber\\
&\leq&|x_{mn}|+|x_{m,n-1}|+|x_{m-1},n|+|x_{m-1,n-1}| \label{1}.
\end{eqnarray}
Then, we see by taking supremum over $m,n\in \mathbb{N}$ in
(\ref{1}) that $\|x\|_{\infty}\leq 4K$, i.e., $x\in
\mathcal{M}_{u}(\Delta)$.

Now, we see that the inclusion is strict. Let $x=(x_{mn})$ be
defined by $$x_{mn}=mn$$ for all $m,n\in \mathbb{N}$. Then the
sequence is in $x\in \mathcal{M}_{u}(\Delta)\backslash
\mathcal{M}_{u}$. This completes the proof of the theorem.
\end{proof}

\begin{lem}\cite[Theorem 1.2]{bafb1}\label{l2.4}
$\mathcal{L}_{u}\subset \mathcal{BS}\subset \mathcal{M}_{u}$
strictly hold.
\end{lem}
\begin{lem}\cite[Theorem 2.9]{bafb1}\label{l2.5}
Let $v\in \{p,bp,r\}$. Then, the inclusion $\mathcal{BV}\subset
\mathcal{C}_{v}$ and $\mathcal{BV}\subset \mathcal{M}_{u}$ strictly
hold.
\end{lem}

Combining Lemma \ref{l2.4}, Lemma \ref{l2.5} and Theorem \ref{t2.3},
we get the following corollaries.

\begin{cor}
The inclusion $\mathcal{L}_{u}\subset \mathcal{BS}\subset
\mathcal{M}_{u}(\Delta)$ strictly hold.
\end{cor}

\begin{cor}
The inclusion $\mathcal{BV}\subset \mathcal{M}_{u}(\Delta)$ strictly
holds.
\end{cor}

\begin{thm}
The inclusion  $\mathcal{L}_{q} \subset \mathcal{L}_{q}(\Delta)$
strictly holds; where $1\leq q<\infty$.
\end{thm}
\begin{proof}
The prove the validity of the inclusion $\mathcal{L}_{q}\subset
\mathcal{L}_{q}(\Delta)$ for $1\leq q<\infty$, it suffices to show
the existence of a number $K>0$ such that
$$
\|x\|_{\mathcal{L}_{q}(\Delta)}\leq K\|x\|_{\mathcal{L}_{q}}
$$
for every $x\in \mathcal{L}_{q}$. Let $x\in \mathcal{L}_{q}$ and
$1\leq q<\infty$. Then, we obtain
\begin{eqnarray*}
\|x\|_{\mathcal{L}_{q}(\Delta)}&=&\bigg\{\sum_{m,n}
|x_{mn}-x_{m,n-1}-x_{m-1,n}+x_{m-1,n-1}|^{q}\bigg\}^{1/q}\\
&\leq& 4\|x\|_{\mathcal{L}_{q}}.
\end{eqnarray*}
This shows that the inclusion $\mathcal{L}_{q}\subset
\mathcal{L}_{q}(\Delta)$ holds.

Additionally, since the sequence $x=(x_{mn})$ defined by
\begin{eqnarray*}
x_{mn}:=\left\{
\begin{array}{rcl}
1&, & n=0 \\
0&,  &\textrm{otherwise} \\
\end{array}
\right.
\end{eqnarray*}
for all $m,n\in \mathbb{N}$ is in $\mathcal{L}_{q}(\Delta)$ but not
in $\mathcal{L}_{q}$, as asserted. This completes the proof.
\end{proof}

\begin{thm}\label{t2.9}
The following statements hold:

(i)  $\mathcal{C}_{p}$ is the subspace of the space
$\mathcal{C}_{p}(\Delta)$.

(ii) $\mathcal{C}_{0p}$ is the subspace of the space
$\mathcal{C}_{0p}(\Delta)$.

(iii) $\mathcal{C}_{bp}$ is the subspace of the space
$\mathcal{C}_{bp}(\Delta)$.

(iv) $\mathcal{C}_{r}$ is the subspace of the space
$\mathcal{C}_{r}(\Delta)$.
\end{thm}
\begin{proof}
We only prove that the inclusion $\mathcal{C}_{p} \subset
\mathcal{C}_{p}(\Delta)$ holds. Let us take $x\in \mathcal{C}_{p}$.
Then, for a given $\varepsilon>0$, there exists an
$n_{x}(\varepsilon)\in \mathbb{N}$ such that
$$
|x_{mn}-l|<\frac{\varepsilon}{4}
$$
for all $m,n>n_{x}(\varepsilon)$. Then,
\begin{eqnarray*}
|(\Delta x)_{mn}|&=&\big|x_{mn}-x_{m,n-1}-x_{m-1,n}+x_{m-1,n-1}\big|\\
&\leq&|x_{mn}-l|+|x_{m,n-1}-l|+|x_{m-1,n}-l|+|x_{m-1,n-1}-l|\\
&<&\frac{\varepsilon}{4}+\frac{\varepsilon}{4}+\frac{\varepsilon}{4}+\frac{\varepsilon}{4}=\varepsilon
\end{eqnarray*}
for sufficiently large $m,n$ which means that $p-\lim(\Delta
x)_{mn}=0$. Hence, $x\in \mathcal{C}_{p}(\Delta)$ that is to say
that $\mathcal{C}_{p} \subset \mathcal{C}_{p}(\Delta)$ holds, as
expected.

Now, we see that the inclusion is strict. Let $x=(x_{mn})$ be
defined by $$x_{mn}=(m+1)(n+1)$$ for all $m,n\in \mathbb{N}$. It is
easy to see that
$$
p-\lim (\Delta x)_{mn}=1.
$$
But
$$
\lim_{m,n\rightarrow \infty}(m+1)(n+1)
$$
which does not tend to a finite limit. Hence $x\notin
\mathcal{C}_{p}$. This completes the proof.
\end{proof}
\begin{lem}\cite[Theorem 2.3]{bafb1}\label{l2.10}
$\mathcal{CS}_{p}$ is the subspace of $\mathcal{C}_{p}$.
\end{lem}

Combining Lemma \ref{l2.10} and Theorem \ref{t2.9}, we get the
following corollary.
\begin{cor}
The inclusion $\mathcal{CS}_{p}\subset \mathcal{C}_{p}(\Delta)$
strictly holds.
\end{cor}

\section{The alpha- and beta-duals of the new spaces of double sequences}

In this section, we determine the alpha-dual of the space
$\mathcal{M}_{u}(\Delta)$ and the $\beta(r)-$dual of the space
$\mathcal{C}_{r}(\Delta)$, and $\beta(\vartheta)-$dual of the space
$\mathcal{C}_{\eta}(\Delta)$ of double sequences, $\vartheta,\eta
\in \{p,bp,r\}$. Although the alpha-dual of a space of double
sequences is unique, its beta-dual may be more than one with respect
to $\vartheta-$convergence.

\begin{thm}
$\{\mathcal{M}_{u}(\Delta)\}^\alpha=\mathcal{L}_u$.
\end{thm}
\begin{proof}
Let $x=(x_{kl})\in \mathcal{M}_{u}(\Delta)$ and $z=(z_{kl})\in
\mathcal{L}_u$. Hence, there is a sequence $y=(y_{ij})\in
\mathcal{M}_{u}$ related with $x=(x_{kl})$ from Theorem \ref{t2.2}
and there is a positive real number $K$ such that
$|y_{ij}|\leq\frac{K}{(k+1)(l+1)}$ for all $i,j\in\mathbb{N}$. So we
use the relation (\ref{2.6})we have that,

$$
\sum_{k,l}|z_{kl}x_{kl}|=\sum_{k,l}\bigg|z_{kl}\sum_{i=0}^k\sum_{j=0}^ly_{ij}\bigg|\leq
K\sum_{k,l}|z_{kl}| <\infty$$ so
$z\in\{\mathcal{M}_{u}(\Delta)\}^\alpha$, that is
\begin{equation}\label{2.7}
\mathcal{L}_u\subset\{\mathcal{M}_{u}(\Delta)\}^\alpha.
\end{equation}

Conversely, suppose that
$z=(z_{kl})\in\{\mathcal{M}_{u}(\Delta)\}^\alpha$, that is
$\sum_{k,l}|z_{kl}x_{kl}|<\infty$ for all $x=(x_{kl})\in
\mathcal{M}_{u}(\Delta)$. If $z=(z_{kl})\notin \mathcal{L}_u$, then
$\sum_{k,l}|z_{kl}|=\infty$. Further, if we choose $y=(y_{kl})$ such
that

\begin{eqnarray*}
y_{kl}:=\left\{
\begin{array}{rcl}
\displaystyle \frac{1}{(k+1)(l+1)}&, &0\leq i\leq k,\quad 0\leq j\leq l \\
\displaystyle 0&,  &\textrm{otherwise} \\
\end{array}
\right.
\end{eqnarray*}
for all $k,l\in \mathbb{N}$. Then, $y\in\mathcal{M}_{u}$  but

$$
\sum_{k,l}|z_{kl}x_{kl}|=\sum_{k,l}\bigg|z_{kl}\sum_{i=0}^k\sum_{j=0}^l\frac{1}{(k+1)(l+1)}\bigg|
=\sum_{k,l}|z_{kl}|=\infty.
$$
Hence, $z\notin\{\mathcal{M}_{u}(\Delta)\}^\alpha$,this is a
contradiction. So, we have the following inclusion,

\begin{equation}\label{2.8}
\{\mathcal{M}_{u}(\Delta)\}^\alpha\subset\mathcal{L}_u.
\end{equation}\\
Hence, from the inclusions (\ref{2.7}) and (\ref{2.8}) we get
$$
\{\mathcal{M}_{u}(\Delta)\}^\alpha=\mathcal{L}_u.
$$
\end{proof}

Now, we may give the $\beta-$duals of the spaces with respect to the
$\vartheta-$convergence using the technique in \cite{bafb2} and
\cite{bafb3} for the single sequences.

The conditions for a 4-dimensional matrix to transform the spaces
$\mathcal{C}_{bp}, \mathcal{C}_{r}$ and $\mathcal{C}_{p}$ into the
space $\mathcal{C}_{bp}$ are well known (see for example
\cite{hjh,mzmmsam}).

\begin{lem}\label{beta-r}
The matrix $A=(a_{mnij})$ is in
$(\mathcal{C}_r:\mathcal{C}_{\vartheta})$ if and only if the
following conditions hold:

\begin{eqnarray}
\sup_{m,n\in \mathbb{N}}\sum_{i,j}|a_{mnij}|<\infty \label{r1},\\
\exists~v\in\mathbb{C}\ni\vartheta-\lim_{m,n\rightarrow\infty}\sum_{i,j}a_{mnij}=v\label{r2},\\
\exists~(a_{ij})\in\Omega\ni\vartheta-\lim_{m,n\rightarrow\infty}a_{mnij}=a_{ij}\quad\textrm{for
all}~\textrm{i,j}\in\mathbb{N}\label{r3},\\
\exists~u^{j_0}\in\mathbb{C}\ni\vartheta-\lim_{m,n\rightarrow\infty}\sum_{i}a_{mnij_{0}}=u^{j_0}~
\textrm{for fixed}~\textrm{j}_0\in\mathbb{N}\label{r4},\\
\exists~v_{i_0}\in\mathbb{C}\ni\vartheta-\lim_{m,n\rightarrow\infty}\sum_{j}a_{mni_{0}j}=v_{i_0}~
\textrm{for fixed}~\textrm{i}_0\in\mathbb{N}\label{r5}.
\end{eqnarray}

\end{lem}

\begin{lem}\label{beta-bp}
The matrix $A=(a_{mnij})$ is in
$(\mathcal{C}_{bp}:\mathcal{C}_{\vartheta})$ if and only if the
conditions (\ref{r1})-(\ref{r3}) of Lemma \ref{beta-r} hold, and
\begin{eqnarray}
\vartheta-\lim_{m,n\rightarrow\infty}\sum_{i}|a_{mnij_0}-a_{ij_0}|=0~~\textrm{for
each fixed}~\textrm{j}_0\in\mathbb{N}\label{bp1},\\
\vartheta-\lim_{m,n\rightarrow\infty}\sum_{j}|a_{mni_0j}-a_{i_0j}|=0~~\textrm{for
each fixed}~\textrm{i}_0\in\mathbb{N}.\label{bp2}
\end{eqnarray}
\end{lem}
\begin{lem}\label{beta-p}
The matrix $A=(a_{mnij})$ is in
$(\mathcal{C}_{p}:\mathcal{C}_{\vartheta})$ if and only if the
conditions (\ref{r1})-(\ref{r3}) of Lemma \ref{beta-r} hold, and
\begin{eqnarray}
\forall~i\in\mathbb{N}~ \exists ~J\in\mathbb{N}\ni
a_{mnij}=0~\textrm{for}~j>J~\textrm{for all m,n}\in\mathbb{N}\label{p1},\\
\forall~j\in\mathbb{N}~ \exists ~I\in\mathbb{N}\ni
a_{mnij}=0~\textrm{for}~i>I~\textrm{for all
m,n}\in\mathbb{N}.\label{p2}
\end{eqnarray}
\end{lem}
\begin{thm}
Define the sets
\begin{eqnarray*}
F_1&=&\bigg\{a=(a_{ij})\in\Omega:\sum_{i,j}(i+1)(j+1)|a_{ij}|<\infty
\bigg\},\\
F_2&=&\bigg\{a=(a_{ij})\in\Omega:r-\lim_{m,n\rightarrow\infty}\sum_{i}^{m}\sum_{p=i}^{m}\sum_{q=j_0}^{n}a_{pq}~~\textrm{exists}~\textrm{for
each fixed}~j_0
\bigg\},\\
F_3&=&\bigg\{a=(a_{ij})\in\Omega:r-\lim_{m,n\rightarrow\infty}\sum_{j}^{m}\sum_{p=i_0}^{m}\sum_{q=j}^{n}a_{pq}~~\textrm{exists}~\textrm{for
each fixed}~i_0 \bigg\}.
\end{eqnarray*}
Then, $\{\mathcal{C}_r(\Delta)\}^{\beta(r)}=F_1\cap F_2\cap F_3$.
\end{thm}
\begin{proof}
Let $x=(x_{ij})\in\mathcal{C}_r(\Delta)$. Then, there exists a
sequence  $y=(y_{mn})\in\mathcal{C}_r$. Consider the following
equality
\begin{eqnarray*}
z_{mn}&=&\sum_{i=0}^{m}\sum_{j=0}^{n}a_{ij}x_{ij}=\sum_{i=0}^{m}\sum_{j=0}^{n}\bigg(\sum_{p=0}^{i}\sum_{q=0}^{j}y_{pq}\bigg)a_{ij}\\
&=&\sum_{i=0}^{m}\sum_{j=0}^{n}\bigg(\sum_{p=i}^{m}\sum_{q=j}^{n}a_{pq}\bigg)y_{ij}\\
&=&\sum_{i=0}^{m}\sum_{j=0}^{n}b_{mnij}y_{ij}=(By)_{ij}
\end{eqnarray*}
for all $m,n\in\mathbb{N}$. Hence we can  define the
four-dimensional matrix $B=(b_{mnij})$ as following
\begin{eqnarray}\label{3.12}
b_{mnij}:=\left\{
\begin{array}{rcl}
\displaystyle\sum_{p=i}^{m}\sum_{q=j}^{n}a_{pq}&, & 0\leq i\leq m, \quad 0\leq j\leq n, \\
\displaystyle 0&,  &\textrm{otherwise}. \\
\end{array}
\right.
\end{eqnarray}

Thus we see that $ax=(a_{mn}x_{mn})\in\mathcal{CS}_r$ whenever
$x=(x_{mn})\in\mathcal{C}_r(\Delta)$ if and only  if
$z=(z_{mn})\in\mathcal{C}_r$ whenever $y=(y_{mn})\in\mathcal{C}_r$.
This means that $a=(a_{mn})\in
\{\mathcal{C}_{r}(\Delta)\}^{\beta(r)}$ if and only if  $B\in
(\mathcal{C}_r:\mathcal{C}_r)$. Therefore, we consider the following
equality and equation
\begin{eqnarray}
\sup_{m,n\in\mathbb{N}}\sum_{i=0}^{m}\sum_{j=0}^{n}|b_{mnij}|\nonumber
&\leq&\sup_{m,n\in\mathbb{N}}\sum_{i=0}^{m}\sum_{j=0}^{n}\bigg(\sum_{p=i}^{m}\sum_{q=j}^{n}|a_{pq}|\bigg)\nonumber\\
&=&\sup_{m,n\in\mathbb{N}}\sum_{i=0}^{m}\sum_{j=0}^{n}\bigg(\sum_{p=0}^{i}\sum_{q=0}^{j}|a_{ij}|\bigg)\nonumber\\
&=&\sup_{m,n\in\mathbb{N}}\sum_{i=0}^{m}\sum_{j=0}^{n}(i+1)(j+1)|a_{ij}|\label{pr1},
\end{eqnarray}
\begin{eqnarray}
r-\lim_{m,n\rightarrow\infty}\sum_{i,j}b_{mnij}\nonumber
&=&r-\lim_{m,n\rightarrow\infty}\sum_{i=0}^{m}\sum_{j=0}^{n}\bigg(\sum_{p=0}^{i}\sum_{q=0}^{j}a_{pq}\bigg)\\
&=&\sum_{i,j}\bigg(\sum_{p=0}^{i}\sum_{q=0}^{j}a_{pq}\bigg).
\end{eqnarray}
Then, we derive from the condition (\ref{r1})-(\ref{r3}) that
\begin{equation}
\sum_{i,j}(i+1)(j+1)|a_{ij}|<\infty.
\end{equation}
Further, from Lemma \ref{beta-r} conditions (\ref{r4}) and
(\ref{r5}),
\begin{equation}
r-\lim_{m,n\rightarrow\infty}\sum_{i}^{m}b_{mnij_0}=r-\lim_{m,n\rightarrow\infty}\sum_{i}^{m}\sum_{p=i}^{m}\sum_{q=j_0}^{n}a_{pq}
\end{equation}
exists for each fixed  $j_0\in\mathbb{N}$ and
\begin{equation}
r-\lim_{m,n\rightarrow\infty}\sum_{j}^{n}b_{mni_0j}=r-\lim_{m,n\rightarrow\infty}\sum_{i}^{m}\sum_{p=i_0}^{m}\sum_{q=j}^{n}a_{pq}
\end{equation}
exists for each fixed  $i_0\in\mathbb{N}$. This show that
$\{\mathcal{C}_r(\Delta)\}^{\beta(r)}=F_1\cap F_2\cap F_3$ which
completes the proof.
\end{proof}

Now, we may give our theorem exhibiting the $\beta(\vartheta)$-dual
of the series space $\mathcal{C}_\eta(\Delta)$ in the case
$\eta,\vartheta\in\{p, bp, r\}$, without proof.
\begin{thm}
$\{\mathcal{C}_\eta(\Delta)\}^{\beta(\vartheta)}=\{a=(a_{mn})\in\Omega:
B=(b_{mnij})\in(\mathcal{C_\eta:\mathcal{C_\vartheta}})\}$ where
$B=(b_{mnij})$ is defined by (\ref{3.12}).
\end{thm}

\section{Characterization of some classes of four dimensional matrices}

In the present section, we characterize the matrix transformations
from the space $\mathcal{C}_{r}(\Delta)$ to the double sequence
space $\mathcal{C}_{\vartheta}$ . Although the theorem
characterizing the class $(\mu:\mathcal{C}_{\vartheta}(\Delta))$ are
stated and proved, the necessary and sufficient conditions on a four
dimensional matrix belonging to the classes
$(\mathcal{C}_{r}:\mathcal{C}_{r}(\Delta))$ and
$(\mathcal{C}_{bp}:\mathcal{C}_{bp}(\Delta))$ also given without
proof.

\begin{thm}\label{t4.1}
$A=(a_{mnkl})\in (\mathcal{C}_{r}(\Delta):\mathcal{C}_{\vartheta})$
if and only if the following conditions hold:
\begin{eqnarray}
\sup_{m,n}\sum_{k,l} \bigg|\sum_{p=k}^{\infty}\sum_{q=l}^{\infty}a_{mnpq}\bigg|<\infty, \label{s1}\\
\vartheta-\lim_{s,t\rightarrow \infty}
\sum_{i=0}^{s}\sum_{p=i}^{s}\sum_{p=j_{0}}^{t}a_{mnpq}-\textrm{exists
for fixed} \ j_{0},\\
\vartheta-\lim_{s,t\rightarrow \infty}
\sum_{j=0}^{t}\sum_{p=i_0}^{s}\sum_{p=j}^{t}a_{mnpq}-\textrm{exists
for fixed} \ i_{0},\\
\vartheta-\lim_{m,n} \sum_{p=k}^{\infty}\sum_{q=l}^{\infty}a_{mnpq}=a_{kl}\  \textrm{for all} \ k,l\in \mathbb{N},\\
\exists u^{l_{0}} \in \mathbb{C}\ni \vartheta-\lim_{m,n}\sum_{k}\sum_{p=k}^{\infty}\sum_{q=l_0}^{\infty}a_{mnpq}=u^{l_{0}} \quad\textrm{for fixed}\ l_{0}\in \mathbb{N},\\
\exists v_{k_{0}} \in \mathbb{C}\ni \vartheta-\lim_{m,n}\sum_{l}\sum_{p=k_{0}}^{\infty}\sum_{q=l}^{\infty}a_{mnpq}=v_{k_{0}} \quad \textrm{for fixed}\ k_{0}\in \mathbb{N} ,\\
\exists v\in \mathbb{C}\ni \vartheta-\lim_{m,n}
\sum_{k,l}\sum_{p=k}^{\infty}\sum_{q=l}^{\infty}a_{mnpq}=v
\label{s7}.
\end{eqnarray}
\end{thm}
\begin{proof}
Let us take any $x=(x_{mn})\in \mathcal{C}_{r}(\Delta)$ and define
the sequence $y=(y_{kl})$ by
$$
y_{kl}=x_{kl}-x_{k-1,l}-x_{k,l-1}+x_{k-1,l-1}; \quad (k,l\in
\mathbb{N}).
$$
Then, $y=(y_{kl})\in \mathcal{C}_{r}$ by Theorem \ref{t2.2}. Now,
for the $(s,t)$th rectangular partial sum of the series $\sum_{j,k}
a_{mnjk}x_{jk}$, we derive that
\begin{eqnarray}\label{4.1}
(Ax)_{mn}^{[s,t]}&=&\sum_{j=0}^{s}\sum_{k=0}^{t}a_{mnjk}x_{jk} \nonumber\\
&=&\sum_{j=0}^{s}\sum_{k=0}^{t}
\bigg(\sum_{p=0}^{j}\sum_{q=0}^{k}y_{pq}\bigg)a_{mnjk}\nonumber\\
&=&\sum_{j=0}^{s}\sum_{k=0}^{t}
\bigg(\sum_{p=j}^{s}\sum_{q=k}^{t}a_{mnpq}\bigg)y_{jk}
\end{eqnarray}
for all $m,n,s,t\in \mathbb{N}$. Define the matrix
$B_{mn}=(b_{mnjk}^{[s,t]})$ by
\begin{eqnarray}
b_{mnjk}^{[s,t]}:=\left\{
\begin{array}{rcl}
\displaystyle\sum_{p=j}^{s}\sum_{q=k}^{t}a_{mnpq}&, & 0\leq j\leq s, \quad 0\leq k\leq t, \\
\displaystyle 0&,  &\textrm{otherwise}. \\
\end{array}
\right.
\end{eqnarray}
Then, the equality (\ref{4.1}) may be rewritten as
\begin{equation}\label{4.3}
(Ax)_{mn}^{[s,t]}=(B_{mn}y)_{[s,t]}.
\end{equation}
Then, the convergence of the rectangular partial sums
$(Ax)_{mn}^{[s,t]}$ in the regular sense for all $m,n\in \mathbb{N}$
and for all $x\in \mathcal{C}_{r}(\Delta)$ is equivalent of saying
that $B_{mn}\in (\mathcal{C}_{r}:\mathcal{C}_{\vartheta})$. Hence,
the following conditions

\begin{eqnarray}
\sum_{k,l}(k+1)(l+1)|a_{mnkl}|<\infty, \label{d1}\\
\vartheta-\lim_{s,t\rightarrow \infty}
\sum_{i=0}^{s}\sum_{p=i}^{s}\sum_{p=j_{0}}^{t}a_{mnpq}-\textrm{exists
for fixed} \ j_{0},\\
\vartheta-\lim_{s,t\rightarrow \infty}
\sum_{j=0}^{t}\sum_{p=i_0}^{s}\sum_{p=j}^{t}a_{mnpq}-\textrm{exists
for fixed} \ i_{0}
\end{eqnarray}
must be satisfied for every fixed  $m,n\in \mathbb{N}$. In this
case,
$$
\vartheta-\lim_{s,t\rightarrow \infty}
b_{mnjk}^{[s,t]}=\sum_{p=j}^{\infty}\sum_{q=k}^{\infty}a_{mnpq},
$$
$$
\vartheta-(Ax)_{mn}^{[s,t]}=r-\lim(B_{mn}y)
$$
hold. Thus, we derive from the two-sided implication "$Ax$ is in
$C_{r}$ whenever $x\in \mathcal{C}_{r}(\Delta)$ if and only if
$B=\big(\sum_{p=j}^{\infty}\sum_{q=k}^{\infty}a_{mnpq}\big)_{mn}\in
(\mathcal{C}_{r}:\mathcal{C}_{\vartheta})$", we have Lemma
\ref{beta-r} that
\begin{eqnarray}
\sup_{m,n}\sum_{k,l} \bigg|\sum_{p=k}^{\infty}\sum_{q=l}^{\infty}a_{mnpq}\bigg|<\infty, \\
\vartheta-\lim_{m,n} \sum_{p=k}^{\infty}\sum_{q=l}^{\infty}a_{mnpq}=a_{kl}\  \textrm{for all} \ k,l\in \mathbb{N},\\
\exists u^{l_{0}} \in \mathbb{C}\ni \vartheta-\lim_{m,n}\sum_{k}\sum_{p=k}^{\infty}\sum_{q=l_0}^{\infty}a_{mnpq}=u^{l_{0}} \quad\textrm{for fixed}\ l_{0}\in \mathbb{N},\\
\exists v_{k_{0}} \in \mathbb{C}\ni \vartheta-\lim_{m,n}\sum_{l}\sum_{p=k_{0}}^{\infty}\sum_{q=l}^{\infty}a_{mnpq}=v_{k_{0}} \quad \textrm{for fixed}\ k_{0}\in \mathbb{N} ,\\
\exists v\in \mathbb{C}\ni \vartheta-\lim_{m,n}
\sum_{k,l}\sum_{p=k}^{\infty}\sum_{q=l}^{\infty}a_{mnpq}=v
\label{d7}.
\end{eqnarray}
Now, from the conditions (\ref{d1})-(\ref{d7}), we have that
$A=(a_{mnkl})\in (\mathcal{C}_{r}(\Delta):\mathcal{C}_{v})$ if and
only if the conditions (\ref{s1})-(\ref{s7}) hold. This completes
the proof.

\end{proof}

\begin{thm}\label{t4.2}
Suppose that the elements of the four dimensional infinite matrices
$E=(e_{mnkl})$ and $F=(f_{mnkl})$ are connected with the relation
\begin{equation}\label{4.4}
f_{mnkl}=\sum_{i=m-1}^{m}\sum_{j=n-1}^{n}(-1)^{m+n-i-j}e_{ijkl}
\end{equation}
for all $k,l,m,n\in \mathbb{N}$ and $\mu$ be any given space of
double sequences. Then, $E\in (\mu:\mathcal{C}_{\vartheta}(\Delta))$
if and only if $F\in (\mu:\mathcal{C}_{\vartheta})$.
\end{thm}
\begin{proof}
Let $x=(x_{kl})\in \mu$ and consider the following equality with
(\ref{4.4})
\begin{equation}\label{4.5}
\sum_{i=m-1}^{m}\sum_{j=n-1}^{n}\sum_{k=s-1}^{s}\sum_{l=t-1}^{t}(-1)^{m+n-i-j}e_{ijkl}x_{kl}=\sum_{k=s-1}^{s}\sum_{l=t-1}^{t}f_{mnkl}x_{kl}
\end{equation}
for all $m,n,s,t\in \mathbb{N}$. By letting $s,t\rightarrow \infty$
in (\ref{4.5}) one can derive that
\begin{equation}\label{4.6}
\sum_{i=m-1}^{m}\sum_{j=n-1}^{n}(-1)^{m+n-i-j}(Ex)_{ij}=(Fx)_{mn}
\quad \textrm{for all} \quad m,n\in \mathbb{N}.
\end{equation}
Therefore, it is seen by (\ref{4.6}) that $Ex\in
\mathcal{C}_{\vartheta}(\Delta)$ if and only if $Fx\in
\mathcal{C}_{\vartheta}$ whenever $x\in \mu$. This step completes
the proof.
\end{proof}

Of course, Theorem \ref{t4.2} has several consequences depending on
the choice  of the sequence space $\mu$. Prior to giving some
results as an application of this idea, we need the following
lemmas:

\begin{lem}\cite{hjh,gmr,mz2}\label{l4.3}
$A=(a_{mnkl})\in (\mathcal{C}_{r}:\mathcal{C}_{r})$ if and only if
\begin{eqnarray}
\sup_{m,n\in \mathbb{N}}\sum_{k,l}|a_{mnkl}|<\infty, \label{4.7}\\
\exists~(a_{kl})\in\Omega\ni
r-\lim_{m,n\rightarrow\infty}a_{mnkl}=a_{kl}\quad\textrm{for
each}~\textrm{k,l}\in\mathbb{N}, \label{4.8}\\
\exists\ v\in \mathbb{C}\ni r-\lim_{m,n\rightarrow \infty}
\sum_{k,l} a_{mnkl}=v, \label{4.9}\\
\exists\ u^{l_{0}},v_{k_{0}}\in \mathbb{C}\ni r-\lim_{m,n\rightarrow
\infty} \sum_{k} a_{mnkl_{0}}=u^{l_{0}} \ \textrm{and} \label{4.10}\\
r-\lim_{m,n\rightarrow \infty} \sum_{l} a_{mnk_{0}l}=v_{k_{0}} \
\textrm{for any} \ k_{0},l_{0}\in \mathbb{N}. \nonumber
\end{eqnarray}
\end{lem}

\begin{lem}\cite{hjh,gmr,mz2}\label{l4.4}
$A=(a_{mnkl})\in (\mathcal{C}_{bp}:\mathcal{C}_{bp})$ if and only if
\begin{eqnarray}
\sup_{m,n\in \mathbb{N}}\sum_{k,l}|a_{mnkl}|<\infty, \label{4.11}\\
\exists~(a_{kl})\in\Omega\ni
bp-\lim_{m,n\rightarrow\infty}a_{mnkl}=a_{kl}\quad\textrm{for
each}~\textrm{k,l}\in\mathbb{N},\label{4.12}\\
\exists\ v\in \mathbb{C}\ni bp-\lim_{m,n\rightarrow \infty}
\sum_{k,l} a_{mnkl}=v,\label{4.13}\\
bp-\lim_{m,n\rightarrow \infty} \sum_{k} |a_{mnkl_{0}}-a_{kl_{0}}|=0
\quad \textrm {and} \label{4.14}\\
bp-\lim_{m,n\rightarrow \infty} \sum_{l} |a_{mnk_{0}l}-a_{k_{0}l}|=0
\quad \textrm{for any}\quad  k_{0},l_{0}\in \mathbb{N}. \nonumber
\end{eqnarray}
\end{lem}
\begin{cor}
Suppose that the relation (\ref{4.4}) holds between the elements of
the four dimensional infinite matrices $E=(e_{mnkl})$ and
$F=(f_{mnkl})$. Then, the following statements hold:\\

(i) $E=(e_{mnjk})\in (\mathcal{C}_{r}:\mathcal{C}_{r}(\Delta))$ if
and only if the conditions (\ref{4.7})-(\ref{4.10}) hold with
$f_{mnkl}$ instead of $a_{mnkl}$.

(ii) $E=(e_{mnjk})\in (\mathcal{C}_{bp}:\mathcal{C}_{bp}(\Delta))$
if and only if the conditions (\ref{4.11})-(\ref{4.14}) hold with
$f_{mnkl}$ instead of $a_{mnkl}$.
\end{cor}

\section*{Acknowledgement}
We wish to thank the referees for their valuable suggestions and
comments which improved the paper considerably.

\vskip 1truecm
\end{document}